\documentclass[12pt,twoside,leqno]{article}
\usepackage[T1]{fontenc}
\usepackage{amsfonts}
\usepackage{amsmath,amsthm}
\usepackage{amssymb,latexsym}
\usepackage{enumerate}

\theoremstyle{plain}
\newtheorem{theorem}{Theorem}[section]
\newtheorem{lemma}[theorem]{Lemma}
\newtheorem{proposition}[theorem]{Proposition}
\newtheorem{corollary}[theorem]{Corollary}

\theoremstyle{definition}
\newtheorem{example}[theorem]{Example}
\newtheorem{definition}[theorem]{Definition}
\theoremstyle{remark}

\DeclareMathOperator{\cl}{{\rm cl}}

\begin{document}
\title{On Urysohn's Lemma for generalized topological spaces in $\mathbf{ZF}$}
\author{Jacek Hejduk and Eliza Wajch\\
Faculty of Mathematics and Computer Science,\\
\L\'od\'z University, Banacha 22, 90-238 \L\'od\'z, Poland.\\
jacek.hejduk@wmii.uni.lodz.pl.\\
Institute of Mathematics\\
Faculty of Exact and Natural Sciences \\
Siedlce University of Natural Sciences and Humanities\\
ul. 3 Maja 54, 08-110 Siedlce, Poland.\\
eliza.wajch@wp.pl}
\maketitle
\begin{abstract}
\medskip
A strong generalized topological space is an ordered pair $\mathbf{X}=\langle X, \mathcal{T}\rangle$ such that $X$ is a set and $\mathcal{T}$ is a collection of subsets of $X$ such that $\emptyset, X\in \mathcal{T}$ and $\mathcal{T}$ is stable under arbitrary unions. A necessary and sufficient condition for a strong generalized topological space $\mathbf{X}$ to satisfy Urysohn's lemma or its appropriate variant is shown in $\mathbf{ZF}$. Notions of a U-normal and an effectively normal generalized topological space are introduced.  It is observed that, in $\mathbf{ZF}+\mathbf{DC}$, every U-normal generalized topological space satisfies Urysohn's lemma. It is shown that every effectively normal generalized topological space satisfies Csasz\'ar's modification of Urysohn's Lemma.  A $\mathbf{ZF}$- example of a strong generalized topological normal space which satisfies the Tietze-Urysohn Extension Theorem and fails to satisfy Urysohn's Lemma is shown. 

\noindent\textit{Mathematics Subject Classification (2010)}:  54A35, 03E35, 03E25, 54A05, 54C30, 54D15  \newline 
\textit{Keywords}: Generalized topology, Urysohn's Lemma, effectively normal space, $\mathbf{ZF}$
\end{abstract}

\section{Introduction}
\label{s1}

The set-theoretic framework for this paper is the Zermelo-Fraenkel system of axioms $\mathbf{ZF}$, so no form of the Axiom of Choice ($\mathbf{AC}$) is assumed. The system $\mathbf{ZF+AC}$ is denoted by $\mathbf{ZFC}$.  The set of all Dedekind-finite ordinal numbers of von Neumann is denoted by $\omega$. Then $0=\emptyset$ and, for every $n\in\omega$, $n+1=n\cup\{n\}$. We put $\mathbb{N}=\omega\setminus\{0\}$. If $X$ is a set, then $[X]^{<\omega}$ stands for the set of all finite subsets of $X$. The power set of $X$ is denoted by $\mathcal{P}(X)$. For sets $X$ and $Y$, the set of all mappings from $X$ to $Y$ is denoted by $Y^X$.

Generalized topological spaces in the style of \cite{cs0} have been studied by many mathematicians. As to our knowledge, generalizations of the classical concept of a topology, such that it is not assumed that finite intersections of open sets are open were considered already in \cite{af}. It has been shown, for instance, in \cite{hl1} and \cite{hl2} recently that generalized topologies that are not topologies can appear in a very natural way in some mathematical problems. Needless to say, the set-theoretic strength of Urysohn's Lemma and the Tietze-Urysohn Extension Theorem for topological spaces is very important (see, e.g., \cite[Forms 78 and 375]{hr}, \cite{hkrr} and \cite{kw}).  In \cite{cs1}, a modification of Urysohn's Lemma for normal generalized topological spaces was obtained.  However, it has been done very little about possible modifications of Urysohn's Lemma and the Tietze-Urysohn Extension Theorem  for generalized topological spaces in $\mathbf{ZF}$. In this article, we introduce and investigate in the absence of the axiom of choice several new concepts relevant to normality, Urysohn's Lemma and the Tietze-Urysohn Extension Theorem for generalized topological spaces in the sense of \cite{cs0}.

 Before we pass to the body of the paper, let us recall several basic definitions and establish notation concerning mainly generalized topologies in the sense of \cite{cs0}.
\begin{definition}
\label{s1d1}
(Cf. \cite{cs0}.)
\begin{enumerate}
\item A \emph{generalized topology} in a set $X$ is a collection $\mu$ of subsets of $X$ such that, for every family $\mathcal{U}\subseteq \mu$, $\bigcup\mathcal{U}\in\mu$. That $\mu$ is a generalized topology in $X$ is abbreviated to: $\mu$ is a GT in $X$.
\item A generalized topology $\mu$ in $X$ is called \emph{strong} if $X\in\mu$.
\item A (\emph{strong}) \emph{generalized topological space}  (in abbreviation: a (strong) GT space) is an ordered pair $\mathbf{X}=\langle X, \mu\rangle$ where $X$ is a set and $\mu$ is a (strong) generalized topology in $X$. 
\end{enumerate}
\end{definition}
If $\mu$ is a generalized topology in $X$, then $\emptyset\in\mu$ because $\emptyset$ is the union of an empty subfamily of $\mu$.
\begin{definition}
\label{s1d2}
Let $\mathbf{X}=\langle X, \mu\rangle$ be a GT space and let $A\subseteq X$.
\begin{enumerate}
\item  The set $A$ is called $\mu$-\emph{open} (respectively, $\mu$-\emph{closed}) if $A\in\mu$ (respectively, $X\setminus A\in\mu$).
\item $\cl_{\mu}(A)$ denotes the intersection of all $\mu$-closed sets containing $A$; that is, $\cl_{\mu}(A)$ is the closure of $A$ in $\mathbf{X}$. Sometimes, we denote $\cl_{\mu}(A)$ by $\cl_{\mathbf{X}}(A)$ or by $\cl_{X}(A)$.
\item For $\mu|_A=\{U\cap A: U\in\mu\}$,  the GT space  $\mathbf{A}=\langle A, \mu|_A\rangle$ is called the \emph{subspace} of $\mathbf{X}$ with the underlying set $A$.
\end{enumerate}
\end{definition}

\begin{definition}
\label{s1d3} 
(Cf. \cite{cs0}.)
Let $\mathbf{X}=\langle X, \mu_X\rangle$ and $\mathbf{Y}=\langle Y, \mu_Y\rangle$ be GT spaces. A mapping $f: X\to Y$ is called:
\begin{enumerate}
 \item $\langle \mu_X, \mu_Y\rangle$-continuous if, for every $V\in\mu_Y$, $f^{-1}[V]\in\mu_X$;
 \item $\langle \mu_X, \mu_Y\rangle$-continuous at a point $x\in X$ if, for every $V\in\mu_Y$ with $f(x)\in V$, there is $U\in\mu_X$ such that $x\in U$ and $f[U]\subseteq V$.
 \end{enumerate}
 \end{definition}

Since definitions of $T_0$, $T_1$, $T_2$, $T_3$ and normal GT spaces are known and the same as analogous definitions for topological spaces, let us not write them down here. We recommend  \cite{en} and \cite{w} as basic textbooks on topological spaces. 

\begin{definition}
\label{s1d4}
\begin{enumerate} 
\item $\tau_n$ denotes the natural topology in $\mathbb{R}$ having the family of all open intervals with rational end-points as a base.
\item For the generalized topology $g\tau_n=\{\emptyset, \mathbb{R}\}\cup\{(-\infty,a): a\in\mathbb{R}\}\cup\{(a, +\infty): a\in\mathbb{R}\}\cup\{(-\infty, a)\cup(b,+\infty): a,b\in\mathbb{R}\text{ and } a<b\}$ in $\mathbb{R}$, $\mathbf{R}=\langle \mathbb{R},g\tau_n\rangle$.
\end{enumerate}  
\end{definition}

In \cite{cs1}, the space $\mathbf{R}$ was used in a version of Urysohn's lemma for GT spaces.

Let us introduce the following new concepts for GT spaces:

\begin{definition}
\label{s1d6}
Let $\mathbf{X}=\langle X, \mu\rangle$ be a GT space.
\begin{enumerate}
\item  $\mathbf{UL(X)}$ is the statement: For every pair $A,B$ of disjoint $\mu$-closed sets, there exists a $\langle\mu, \tau_n\rangle$-continuous function $f: X\to\mathbb{R}$ such that $A\subseteq f^{-1}[\{0\}]$ and $B\subseteq f^{-1}[\{1\}]$. If $\mathbf{UL(X)}$ is true, we say that $\mathbf{X}$ satisfies Urysohn's Lemma.
\item $\mathbf{GUL(X)}$ is the statement: For every pair $A,B$ of disjoint $\mu$-closed sets, there exists a $\langle\mu, g\tau_n\rangle$-continuous function $f: X\to\mathbb{R}$ such that $A\subseteq f^{-1}[\{0\}]$ and $B\subseteq f^{-1}[\{1\}]$.
\item $\mathbf{TET(X)}$ is the statement: For any $\mu$- closed set $A$ and any $\langle \mu|_A, \tau_n\rangle$-continuous function $f: A\to\mathbb{R}$, there exists a $\langle \mu, \tau_n\rangle$-continuous function $\tilde{f}:X\to\mathbb{R}$ such that, for every $x\in A$,  $\tilde{f}(x)=f(x)$. If $\mathbf{TET(X)}$ is true, we say that $\mathbf{X}$ satisfies the Tietze-Urysohn Extension Theorem.
\item $\mathbf{GTET(X)}$ is the statement: For any $\mu$-closed set $A$ and any $\langle \mu|_A, g\tau_n\rangle$-continuous function $f: A\to\mathbb{R}$, there exists a $\langle \mu, g\tau_n\rangle$-continuous function $\tilde{f}:X\to\mathbb{R}$ such that, for every $x\in A$,  $\tilde{f}(x)=f(x)$.
\end{enumerate}
\end{definition}

The following proposition is straightforward:

\begin{proposition}
\label{s1p7}
$[\mathbf{ZF}]$ Let $\mathbf{X}=\langle X, \mu\rangle$ be a GT space. 
\begin{enumerate}
\item $\mathbf{UL(X)}$ implies $\mathbf{GUL(X)}$.
\item If $\mathbf{X}$ is a topological space, then $\mathbf{UL(X)}$ and $\mathbf{GUL(X)}$ are equivalent, $\mathbf{TET(X)}$ and $\mathbf{GTET(X)}$ are equivalent, and $\mathbf{TET(X)}$ implies $\mathbf{UL(X)}$.
\item (Cf. \cite[Theorem 3.4]{cs1}.) If $\mathbf{GUL(X)}$ holds, then $\mathbf{X}$ is normal. 
\end{enumerate}
\end{proposition}

For a topological space $\mathbf{X}$, the notation $\mathbf{UL(X)}$ and $\mathbf{TET(X)}$ was established in \cite[Remark 2,3]{kw}.

In the following definition, we recall three known froms from \cite{hr} and introduce a new one.

\begin{definition} 
\label{s1d8}
\begin{enumerate}
\item $\mathbf{DC}$ (the Principle of Dependent Choices, \cite[Form 43]{hr}): For every non-empty set $A$ and every binary relation $S$ on $A$, the following implication holds:  $((\forall x\in A)(\exists y\in A) \langle x, y\rangle\in S)\rightarrow ((\exists a\in A^{\omega})(\forall n\in\omega)\langle a(n), a(n+1)\rangle\in S).$
\item $\mathbf{UL}$ (Urysohn's Lemma, \cite[Form 78]{hr}): For every normal topological space $\mathbf{X}$, $\mathbf{UL(X)}$ holds.
\item $\mathbf{GUL}$: For every normal GT space $\mathbf{X}$, $\mathbf{GUL(X)}$ holds.
\item $\mathbf{TET}$ (the Tietze-Urysohn Extension Theorem, \cite[Form 375]{hr}): For every normal topological space $\mathbf{X}$, $\mathbf{TET(X)}$ holds.
\end{enumerate}
\end{definition}

In \cite{cs1}, \'A.  Csasz\'ar proved in $\mathbf{ZFC}$ the following version of Urysohn's Lemma:

\begin{theorem}
\label{s1t5} 
(Cf. \cite[Theorem 3.3]{cs1}.) $\mathbf{[ZFC]}$ If $\mathbf{X}=\langle X, \mu\rangle$ is a normal GT space, then $\mathbf{GUL(X)}$ holds.
\end{theorem}

In \cite{L} (cf. also \cite{gt} and \cite{hkrr}), it was proved that it is consistent with $\mathbf{ZF}$ the existence of a normal topological space $\mathbf{X}$ for which $\mathbf{UL(X)}$ is false. This implies that  Theorem \ref{s1t5} is not a theorem of $\mathbf{ZF}$. On the other hand, it is well known that, in $\mathbf{ZF}$, $\mathbf{DC}$ implies $\mathbf{TET}$ and, in consequence, $\mathbf{UL}$ also follows from $\mathbf{DC}$ (cf. \cite[entries (43, 78) and (43, 375), pages 339 and 386]{hr}). It was shown in \cite{hkrr} that there is a model of $\mathbf{ZF}$ in which a compact Tychonoff space $\mathbf{X}$ fails to satisfy $\mathbf{TET(X)}$. Therefore, in general, for a compact Hausdorff space $\mathbf{X}$,  $\mathbf{UL(X)}$ need not imply $\mathbf{TET(X)}$ in a model of $\mathbf{ZF}$.

In this article, we show that, in general, for a normal GT space $\mathbf{X}$, $\mathbf{GUL(X)}$ need not imply $\mathbf{UL(X)}$ in $\mathbf{ZF}$. We observe that, in $\mathbf{ZF+DC}$, every normal GT space $\mathbf{X}$ satisfies $\mathbf{GUL(X)}$. We introduce a concept of an effectively normal GT space and prove that, for every effectively normal GT space $\mathbf{X}$, $\mathbf{GUL(X)}$ holds in $\mathbf{ZF}$. We show in $\mathbf{ZF}$ necessary and sufficient conditions for a normal GT space to satisfy $\mathbf{GUL(X)}$, as well as more complicated necessary and sufficient conditions for $\mathbf{X}$ to satisfy $\mathbf{UL(X)}$. Some of the sufficient conditions are shown to be also necessary. Furthermore, we discuss $\mathbf{TET(X)}$ and $\mathbf{GTET(X)}$. We show that there is in $\mathbf{ZF}$ a GT space $\mathbf{X}$ which satisfies the conjunction $\mathbf{TET(X)}\wedge \neg\mathbf{UL(X)}$.

\section{The GT space $\mathbf{R}$}
\label{s2}

In this section, we concentrate on the space GT space $\mathbf{R}=\langle \mathbb{R}, g\tau_n\rangle$ (see Definition \ref{s1d4}(2)). We show in $\mathbf{ZF}$ that $\mathbf{UL(R)}$ is false but $\mathbf{GUL(R)}$, $\mathbf{GTET(R)}$ and $\mathbf{TET(R)}$ are all true. 

To begin, let us observe that $\mathcal{D}(\mathbf{R})=\{\emptyset, \mathbb{R}\}\cup\{(-\infty, a]: a\in\mathbb{R}\}\cup\{[a,+\infty): a\in\mathbb{R}\}\cup\{ [a,b]: a,b\in\mathbb{R}\text{ and } a\leq b\}$ is the collection of all $g\tau_n$-closed sets. If $A,B$ is a pair of non-empty disjoint $g\tau_n$-closed sets, then there exists $c\in\mathbb{R}\setminus(A\cup B)$ such that $A\subseteq (-\infty, c)$ and $B\subseteq (c,+\infty)$ or $B\subseteq (-\infty, c)$ and $A\subseteq (c, \infty)$. This simple observation shows that the space $\mathbf{R}$ is normal.

\begin{lemma}
\label{s2lem1}
$[\mathbf{ZF}]$ Let $f:\mathbb{R}\to\mathbb{R}$ be a $\langle g\tau_n, \tau_n\rangle$-continuous function. Then the set $f[\mathbb{R}]$ is $\tau_n$-connected. 
\end{lemma}
\begin{proof}
Since $\mathcal{D}(\mathbf{R})\cap g\tau_n=\{\emptyset, \mathbb{R}\}$, the GT space $\mathbf{R}$ is connected. Thus, the set $f[\mathbb{R}]$ is connected in $\langle \mathbb{R}, \tau_n\rangle$. 
\end{proof}

\begin{proposition}
$[\mathbf{ZF}]$ $\mathbf{UL(R)}$ is false.
\end{proposition}
\label{s2p2}

\begin{proof}
Suppose that $\mathbf{UL(R)}$ is true. Since the sets $A=[0,1]$ and $B=[2,3]$ are both $g\tau_n$-closed and $A\cap B=\emptyset$, by $\mathbf{UL(R)}$, there exists a $\langle g\tau_n, \tau_n\rangle$-continuous function $f:\mathbb{R}\to\mathbb{R}$ such that $A\subseteq f^{-1}[\{0\}]$ and $B\subseteq f^{-1}[\{1\}]$. It follows from Lemma \ref{s2lem1} that $[0,1]\subseteq f[\mathbb{R}]$.  Let $U=f^{-1}[(-1, \frac{1}{4})]$, $V=f^{-1}[(\frac{1}{3}, \frac{2}{3})]$ and $W=f^{-1}[(\frac{3}{4}, +\infty)]$. The sets $U,V,W$ are pairwise disjoint, non-empty and $g\tau_n$-open. This is impossible. Hence $\mathbf{UL(R)}$ is false.
\end{proof}

\begin{proposition}
\label{s2p3}
$[\mathbf{ZF}]$ $\mathbf{GUL(R)}$ is true. 
\end{proposition}
\begin{proof}
Consider an arbitrary pair $A,B$ of non-empty disjoint $g\tau_n$-closed sets. There exist real numbers $c, d$ such that $c<d$ and either $A\subseteq (-\infty, c]$ and $B\subseteq [d, +\infty)$ or $B\subseteq (-\infty, c]$ and $A\subseteq [d, +\infty)$. We may assume that $A\subseteq (-\infty, c]$ and $B\subseteq [d, +\infty)$. Then we define a function $f:\mathbb{R}\to\mathbb{R}$ as follows:
$$ f(x)=\begin{cases} 0 &\text{if $x\in (-\infty,c)$;}\\
1 &\text{if $x\in [d,+\infty)$;}\\
\frac{x-c}{d-c} &\text{if $x\in (c,d)$.}
\end{cases}
$$
The function $f$ is $\langle g\tau_n, g\tau_n\rangle$-continuous, $A\subseteq f^{-1}[\{0\}]$ and $B\subseteq f^{-1}[\{1\}]$.  
\end{proof}

\begin{proposition}
\label{s2p4}
$[\mathbf{ZF}]$ Both $\mathbf{GTET(R)}$ and $\mathbf{TET(R)}$ are true.
\end{proposition}
\begin{proof}
Let us show that $\mathbf{TET(R)}$ is true.

Let $P$ be a non-empty $g\tau_n$-closed set such that $P\neq\mathbb{R}$ and $P$ is not a singleton. Let $f:P\to\mathbb{R}$ be a $\langle g\tau_n|_P,\tau_n\rangle$-continuous function. 

Suppose that $P=[a, b]$ for some $a,b\in\mathbb{R}$ such that $a<b$. We define a function $f_1:\mathbb{R}\to\mathbb{R}$ as follows:
$$ f_1(x)=\begin{cases} f(x) &\text{if $x\in P$;}\\
f(a) &\text{if $x\in (-\infty, a)$;}\\
f(b) &\text{if  $x\in (b, +\infty)$.}\end{cases}
$$
To show that $f_1$ is $\langle g\tau_n, \tau_n\rangle$-continuous, we consider any set $V\in\tau_n$ and put $U=f^{-1}[V]$. Then $U\in g\tau_n|_P$.  If $U=[a, c)$ for some $c\in (a, b]$, then  $f_1^{-1}[V]= (-\infty, c)\in\tau_n$. If $U=(d, b]$ for some $d\in[a, b)$, then $f_1^{-1}[V]=(d, +\infty)\in g\tau_n$. If $U=[a, c)\cup(d,b]$ for some $c,d\in[a,b]$ such that $a<c\leq d<b$, then $f_1^{-1}[V]=(-\infty,c)\cup(d,+\infty)\in g\tau_n$. If $U=\emptyset$, then $f_1^{-1}[V]=\emptyset\in g\tau_n$. If $U=P$, then $f_1^{-1}[V]=\mathbb{R}\in g\tau_n$. Hence $f_1$ is $\langle g\tau_n, \tau_n\rangle$-continuous. Of course, $f_1(x)=f(x)$ for every $x\in P$.

Suppose that $P=(-\infty, a]$ for some $a\in\mathbb{R}$. In this case, we define a $\langle g\tau_n,\tau_n\rangle$-continuous function $f_2:\mathbb{R}\to\mathbb{R}$ as follows:
$$ f_2(x)=\begin{cases} f(x) &\text{if $x\in P$;}\\
f(a) &\text{if $x\in (a, +\infty)$.}\\
\end{cases}
$$
Suppose that $P=[b,+\infty)$ for some $b\in\mathbb{R}$. In this case, we define a $\langle g\tau_n,\tau_n\rangle$-continuous function $f_3:\mathbb{R}\to\mathbb{R}$ as follows:
$$ f_3(x)=\begin{cases} f(x) &\text{if $x\in P$;}\\
f(b) &\text{if $x\in (-\infty, b)$.}\\
\end{cases}
$$
All this taken together shows that $\mathbf{TET(R)}$ is true.  Using similar arguments, one can check that $\mathbf{GTET(R)}$ is also true.
\end{proof}

\begin{corollary}
\label{s2c5}
 In $\mathbf{ZF}$, for a GT space $\mathbf{X}$, $\mathbf{GUL(X)}$ need not imply $\mathbf{UL(X)}$, and $\mathbf{TET(X)}$ need not imply $\mathbf{UL(X)}$. 
\end{corollary}

\section{Conditions under which $\mathbf{GUL(X)}$ holds}
\label{s3}

Since, for every GT space $\mathbf{X}=\langle X, \mu\rangle$ such that $X\notin\mu$, $\mathbf{UL(X)}$ holds because there does not exist a pair $A,B$ of disjoint $\mu$-closed sets (see \cite[Proposition 2.1]{cs1}), we are concerned mainly with strong GT spaces. 

\begin{proposition}
\label{s3p1}
$[\mathbf{ZF}]$ Let $\mathbf{X}=\langle X, \mu\rangle$ be a strong GT space. Then $\mathbf{GUL(X)}$ holds if and only if, for every pair $A,B$ of non-empty disjoint $\mu$-closed sets there exists a collection $\{U_r: r\in\mathbb{Q}\cap(0,1)\}$ of $\mu$-open sets such that:
\begin{enumerate}
\item[(i)] for all $r,s\in\mathbb{Q}\cap(0,1)$, if $r<s$, then $\cl_{\mu}(U_r)\subseteq U_s$;
\item[(ii)] for every $r\in\mathbb{Q}\cap(0,1)$, $A\subseteq U_r$;
\item[(iii)] for every $r\in\mathbb{Q}\cap(0,1)$, $\cl_{\mu}(U_r)\subseteq X\setminus B$.
\end{enumerate}
\end{proposition}
\begin{proof} Assume that $A,B$ is a pair of non-empty disjoint $\mu$-closed sets.

\emph{Sufficiency.} Suppose that $\{U_r: r\in\mathbb{Q}\cap (0,1)\}$ is a family of $\mu$-open sets satisfying conditions (i)--(iii). For every $r\in\mathbb{Q}$ such that $r\leq 0$, we put $U_r=\emptyset$. For every $r\in\mathbb{Q}$ such that $r>1$, we put $U_r=X$. We also put $U_1=X\setminus B$. Then, as in the standard proof of Urysohn's lemma, we define a function $f:X\to\mathbb{R}$ by putting $f(x)=\inf\{r\in\mathbb{Q}: x\in U_r\}$. Arguing in much the same way, as in the proof to Theorem 3.3 in \cite{cs1}, one can check that the function $f$ is $\langle \mu, g\tau_n\rangle$-continuous, $A\subseteq f^{-1}[\{0\}]$ and $B\subseteq f^{-1}[\{1\}]$.

\emph{Necessity.} Suppose that $g:X\to\mathbb{R}$ is a $\langle\mu, g\tau_n\rangle$-continuous function such that $A\subseteq g^{-1}[\{0\}]$ and $B\subseteq g^{-1}[\{1\}]$. Then, for every $r\in\mathbb{Q}\cap(0,1)$, we define $U_r=g^{-1}[(-\infty, r)]$. 
\end{proof}

That it holds in $\mathbf{ZF+DC}$ that every normal topological space satisfies Urysohn's Lemma is shown in \cite[Problem 2.26]{je} (see also \cite[p. 339]{hr}) but without any detailed proof. Therefore, let us sketch a proof to the following more general theorem for completeness.

\begin{theorem}
\label{s3t2} 
$[\mathbf{ZF}]$ $\mathbf{DC}$ implies that, for every normal GT space $\mathbf{X}$, $\mathbf{GUL(X)}$ holds.
\end{theorem}
\begin{proof}
Let $\mathbf{X}=\langle X, \mu\rangle$ be a strong normal GT space. Suppose that $A,B$ is a pair of non-empty disjoint $\mu$-closed sets. Fix a bijection $\psi:\omega\to\mathbb{Q}\cap(0,1)$. For every $n\in\omega$, let $\mathcal{U}_n$ be a family of all finite sequences $(U_{\psi(i)})_{i\in n+1}$ of $\mu$-open sets such that:
\begin{enumerate}
\item[(i)] for all $i,j\in n+1$, if $\psi(i)<\psi(j)$, then $\cl_{\mu}(U_{\psi(i)})\subseteq U_{\psi(j)}$;
\item[(ii)] for every $i\in n+1$, $A\subseteq U_{\psi(i)}$;
\item[(iii)] for every $i\in n+1$, $\cl_{\mu}(U_{\psi(i)})\subseteq X\setminus B$.
\end{enumerate}
It follows from the normality of $\mathbf{X}$ that, for every $n\in\omega$, $\mathcal{U}_n\neq\emptyset$. Let $\mathcal{U}=\bigcup\limits_{n\in\omega}\mathcal{U}_n$. We define a binary relation $R$ on $\mathcal{U}$ as follows: if $V_1\in\mathcal{U}_n$ with $V_1=(U(1)_{\psi(i)})_{ i\in n+1}$, and  $V_2\in\mathcal{U}_m$ with $V_2=(U(2)_{\psi(i)})_{ i\in m+1}$,  then $\langle V_1,V_2\rangle\in R$ if and only if $n<m$ and, for every $i\in n+1$, $U(1)_{\psi(i)}=U(2)_{\psi(i)}$. It follows from the normality of $\mathbf{X}$ that, for every $n\in\omega$ and every $V\in\mathcal{U}_n$, there exists $W\in\mathcal{U}_{n+1}$ such that $\langle V,W\rangle\in R$. Assuming $\mathbf{DC}$, we may fix $V\in\mathcal{U}^{\omega}$ such that, for every $n\in\omega$, $\langle V(n), V(n+1)\rangle\in R$. Using $V$, one can easily define a family $\{U_r: r\in\mathbb{Q}\cap(0,1)\}$ of $\mu$-open sets satisfying conditions (i)--(iii) of Proposition \ref{s3p1}. Hence $\mathbf{GUL(X)}$ holds in $\mathbf{ZF+DC}$.
\end{proof}

The notion of an effectively normal topological space was introduced in \cite{M} (see also \cite[Note 71]{hr} and \cite{hkrr} for a definition of an effectively normal space).  It is known from \cite{hkrr} that every effectively normal topological space satisfies Urysohn's lemma in $\mathbf{ZF}$. Let us adopt the concept of effective normality to generalized topological spaces.

\begin{definition}
\label{s3d3}
Let $\mathbf{X}=\langle X, \mu\rangle$ be a GT space and let  
$$\mathcal{E}(\mathbf{X})=\{\langle A, B\rangle: A,B \text{ are $\mu$-closed and } A\cap B=\emptyset\},$$
$$\mathcal{O}(\mathbf{X})=\{\langle U,V\rangle: U,V \text{ are $\mu$-open and } U\cap V=\emptyset\}.$$ 
We say that $\mathbf{X}$ is \emph{effectively normal} if there exists a function $F:\mathcal{E}(\mathbf{X})\to\mathcal{O}(\mathbf{X})$ such that, for every $\langle A, B\rangle\in\mathcal{E}(\mathbf{X})$, if $F(\langle A, B\rangle)=\langle U, V\rangle$, then $A\subseteq U$ and $B\subseteq V$.
\end{definition}

\begin{theorem}
\label{s3t4}
$[\mathbf{ZF}]$ For every effectively normal GT space $\mathbf{X}$, $\mathbf{GUL(X)}$ holds.
\end{theorem}
\begin{proof}
Let $\mathbf{X}=\langle X, \mu\rangle$ be an effectively normal GT space. Let $F:\mathcal{E}(\mathbf{X})\to\mathcal{O}(\mathbf{X})$ be a function such that, for every $\langle A, B\rangle\in\mathcal{E}(\mathbf{X})$, if $F(\langle A, B\rangle)=\langle U, V\rangle$, then $A\subseteq U$ and $B\subseteq V$. Let $\psi:\omega\to\mathbb{Q}\cap (0,1)$ be a bijection. Given a pair $A,B$ of non-empty disjoint $\mu$-closed sets,  we can mimic the standard $\mathbf{ZFC}$-proof of Urysohn's Lemma and, by using $F$, we can inductively define in $\mathbf{ZF}$ a family $\{U_r: r\in\mathbb{Q}\cap(0,1)\}$ of $\mu$-open sets satisfying conditions (i)--(iii) of Proposition \ref{s3p1}. Hence $\mathbf{GUL(X)}$ is true in $\mathbf{ZF}$ by Proposition \ref{s3p1}.  
\end{proof}

\begin{proposition}
\label{s3p5}
$[\mathbf{ZF}]$ The space $\mathbf{R}$ is effectively normal.
\end{proposition}
\begin{proof}
Let $\psi:\omega\to \mathbb{Q}$ be a bijection. For every $i\in\omega$, let $U_i=(-\infty, \psi(i))$ and $V_i=(\psi(i), +\infty)$.  To define a mapping $F:\mathcal{E}(\mathbf{R})\to\mathcal{O}(\mathbf{R})$ showing the effective normality of $\mathbf{R}$, for a fixed pair $A,B$ of non-empty disjoint $g\tau_n$-closed sets, we consider the set $N(A,B)=\{i\in\omega: A\subseteq U_i\text{ and }B\subseteq V_i\text{ or } A\subseteq V_i\text{ and }B\subseteq U_i\}$. Clearly, $N(A,B)\neq\emptyset$, so we may define $n(A,B)=\min N(A,B)$ and
$$
F(\langle A, B\rangle)=\begin{cases} \langle U_{n(A,B)}, V_{n(A,B)}\rangle &\text{if } A\subseteq U_{n(A,B)};\\
\langle V_{n(A,B)}, U_{n(A,B)}\rangle &\text{otherwise.}
\end{cases}
$$
If $A,B$ is a pair of $g\tau_n$-closed sets such that either $A$ or $B$ is empty, we put 
$$
F(\langle A, B\rangle)=\begin{cases} \langle \emptyset, \mathbb{R}\rangle &\text{if } A=\emptyset;\\
\langle \mathbb{R}, \emptyset\rangle &\text{if } B=\emptyset.
\end{cases}
$$
\end{proof}

\section{Conditions under which $\mathbf{UL(X)}$ holds}
\label{s4}

Reasonable necessary and sufficient conditions for a GT space  to satisfy $\mathbf{UL(X)}$ are more complicated than the ones to satisfy $\mathbf{GUL(X)}$.

\begin{proposition} 
\label{s4p1}
$[\mathbf{ZF}]$ Let $\mathbf{X}=\langle X, \mu\rangle$ be a GT space. Then, for every pair $A,B$ of non-empty  subsets of $X$, the following conditions are equivalent:
\begin{enumerate}
\item[(a)] there exists a $\langle\mu,\tau_n\rangle$-continuous function $f: X\to\mathbb{R}$ such that $A\subseteq f^{-1}[\{0\}]$ and $B\subseteq f^{-1}[\{1\}]$;
\item[(b)] there exists a family of ordered pairs $\{\langle U_r, F_r\rangle: r\in\mathbb{Q}\cap(0,1)\}$ which satisfies the following conditions:
\begin{enumerate}
\item[(i)] for every $r\in\mathbb{Q}\cap (0,1)$, $U_r$ is a $\mu$-open set, $F_r$ is a $\mu$-closed set and $A\subseteq U_r\subseteq F_r$;
\item[(ii)] for every pair $s,r\in\mathbb{Q}\cap (0,1)$, if $s<r$, then $F_s\subseteq U_r$ and $U_r\setminus F_s\in\mu$;
\item [(iii)] for every $r\in\mathbb{Q}\cap (0,1)$, $F_r\cap B=\emptyset$.
\end{enumerate}
\end{enumerate}
\end{proposition}
\begin{proof}
$(a)\rightarrow (b)$ Given a $\langle \mu, \tau_n\rangle$-continuous function $f: X\to\mathbb{R}$ such that $A\subseteq f^{-1}[\{0\}]$ and $B\subseteq f^{-1}[\{1\}]$, for every $r\in\mathbb{Q}\cap (0,1)$, we can define $U_r=f^{-1}[(-\infty, r)]$ and $F_r=f^{-1}[(-\infty, r]]$. 

$(b)\rightarrow (a)$ Suppose that for a pair $A,B$ of non-empty subsets of $X$, we a given a family $\{\langle U_r, F_r\rangle: r\in\mathbb{Q}\cap(0,1)\}$ which satisfies conditions (i)--(iii). For every non-negative rational number $r$, we define $U_r=F_r=\emptyset$. For every rational number $r\geq 1$, we define $U_r=F_r=X$. We define a function $f:X\to\mathbb{R}$ by putting $f(x)=\inf\{r\in\mathbb{Q}: x\in U_r\}$. As in the standard proof of Urysohn's lemma, one can check that it follows from $(ii)$ that $f$ is $\langle \mu, \tau_n\rangle$-continuous. Clearly, $A\subseteq f^{-1}[\{0\}]$ and $B\subseteq f^{-1}[\{1\}]$.
\end{proof}

\begin{definition}
\label{s4d1}
We say that a GT space $\mathbf{X}=\langle X, \mu\rangle$ is \emph{U-normal} if, for every pair $A,B$ of non-empty disjoint $\mu$-closed sets and for every $n\in\omega$, there exists a family $\{\langle U_i, F_i\rangle : i\in n+1\}$ which satisfies the following conditions:
\begin{enumerate}
\item[(i)] for every $i\in n$, $A\subseteq U_i\subseteq F_i\subseteq U_{i+1}\subseteq F_{i+1}\subseteq X\setminus B$;
\item [(ii)] for every pair $i,j\in n+1$ such that $i\in j$, $U_j\setminus F_i\in\mu$;
\item[(iii)] for every $i\in n+1$, there exist a $\mu$-open set $U$ and a $\mu$-closed set $F$ such that $U\subseteq F$ and if $i=0$, that $F\subseteq U_0$, if $i=n$, then $F_n\subseteq U$, if $0\in i\in n$, then $F_i\subseteq U$ and $F\subseteq U_{i+1}$ and, moreover, for every $j\in n+1$, if $F\subseteq U_j$, then $U_j\setminus F\in \mu$, and if $F_j\subseteq U$, then $U\setminus F_j\in\mu$.
\end{enumerate}
\end{definition}

\begin{definition} Let $\mathbf{X}=\langle X, \mu\rangle$ be a U-normal GT-space. Assume that  $A,B$ is a pair of dsjoint non-empty $\mu$-closed sets. 
\begin{enumerate}
\item[(a)] Let $n\in\omega$ and let $S$ be a set which has exactly $n+1$ elements. A family $\mathcal{U}=\{\langle U_s, F_s\rangle: s\in S\}$ will be called a \emph{U-family of length} $n+1$ for the pair $\langle A, B\rangle$ if there is a bijection $\psi: n+1\to S$ such that the family $\{\langle U_{\psi(i)},F_{\psi(i)}\rangle :i\in n+1\}$ satisfies conditions (i)-(iii) of Definition \ref{s4d1}.
\item[(b)] Suppose that $m,n \in\omega$, $\mathcal{U}=\{\langle U_s, F_s\rangle: s\in S\}$ is a U-family of length $n+1$ for $\langle A,B\rangle$, and $\mathcal{V}=\{\langle V_t, H_t\rangle: t\in T\}$ is a U-family of length $m+1$ for $\langle A,B\rangle$. Then we say that $\mathcal{U}$ is an \emph{extension} of $\mathcal{V}$ if $m\in n$,  $T\subseteq S$ and, for every $t\in T$, $U_t=V_t$ and $F_t= H_t$.
\end{enumerate}
\end{definition}

\begin{theorem}
\label{s4t1} 
It holds in $\mathbf{ZF}$ that $\mathbf{DC}$ implies that every U-normal GT space $\mathbf{X}$ satisfies $\mathbf{UL(X)}$.
\end{theorem}
\begin{proof}
Let $\mathbf{X}=\langle X, \mu\rangle$ be a U-normal GT space. Suppose that $A,B$ is a pair of disjoint non-empty $\mu$-closed sets. Let $\psi:\omega\to\mathbb{Q}\cap (0,1)$ be a bijection.  For every $n\in\omega$, let $\mathcal{U}_n(\langle A, B\rangle)$ be a collection of all U-families $\{\langle U_{\psi{i}}, F_{\psi(i)}\rangle: i\in n+1\}$  of length $n+1$ for $\langle A, B\rangle$. Let $\mathcal{U}(\langle A, B\rangle)=\bigcup\limits_{n\in\omega}\mathcal{U}_n(\langle A, B\rangle)$. We define a binary relation $R$ on $\mathcal{U}(\langle A,B\rangle)$ as follows. If $\mathcal{U},\mathcal{V}\in\mathcal{U}(\langle A,B\rangle$, then $\langle \mathcal{U},\mathcal{V}\rangle\in R$ if and only if $\mathcal{V}$ is an extension of $\mathcal{U}$. It follows from the U-normality of $\mathbf{X}$ that, for every $n\in\omega$, $\mathcal{U}_n(\langle A,B\rangle)\neq\emptyset$ and, for every $\mathcal{U}\in\mathcal{U}(\langle A,B\rangle)$, there exists $\mathcal{V}\in\mathcal{U}(\langle A,B\rangle)$ such that $\langle\mathcal{U}, \mathcal{V}\rangle\in R$. Assuming $\mathbf{DC}$, we can fix a function $H\in\mathcal{U}(\langle A, B\rangle)^{\omega}$ such that, for every $n\in\omega$, $\langle H(n), H(n+1)\rangle\in R$. By using $H$, one can easily define a family $\{\langle U_r, F_r\rangle: r\in\mathbb{Q}\cap(0,1)\}$ satisfying conditions (i)-(iii) of Proposition \ref{s4p1}(b). This, together with Proposition \ref{s4p1}, completes the proof.
\end{proof}

\section{On the topology generated by a generalized topology}
\label{s5}

For a generalized topology $\mu$ in a set $X$, let $\tau(\mu)$ be the topology in $X$ generated by $\mu$. If $\mu$ is strong, then 
$$\tau(\mu)=\{ V\subseteq X: (\forall x\in V)(\exists \mathcal{U}\in [\mu]^{<\omega})( x\in\bigcap\mathcal{U}\subseteq V)\}.$$

The main aim of this section is to show in $\mathbf{ZF}$ an example of a Hausdorff strong GT space $\mathbf{X}=\langle X,\mu\rangle$ such that $\mathbf{GUL(X)}$ is true but $\mathbf{UL}(\langle X,\tau(\mu)\rangle)$ is false. The following theorem shows that, in such an example, $\tau(\mu)$ cannot be compact.

\begin{theorem}
\label{s5t1}
$[\mathbf{ZF}]$ Let $\mathbf{X}=\langle X, \mu\rangle$ be a  strong GT space such that $\mathbf{GUL(X)}$ is true and $\tau(\mu)$ is compact. Then $\mathbf{UL}(\langle X, \tau(\mu)\rangle)$ is true. 
\end{theorem} 
\begin{proof}
Consider any pair $A_0, A_1$ of disjoint non-empty $\tau(\mu)$-closed sets. If $i\in\{0,1\}$, there exist a non-empty set $S_i$ and a family $\{\mathcal{F}_{i,s}: s\in S_i\}$ such that, for every $s\in S_i$, $\mathcal{F}_{i,s}$ is a finite family of $\mu$-closed sets and $A_i=\bigcap\limits_{s\in S_i}(\bigcup\mathcal{F}_{i,s})$. Since $A_0\cap A_1=\emptyset$ and the sets $A_0, A_1$ are both $\tau(\mu)$-compact, for every $i\in\{0,1\}$, there exists a finite subset $K_i$ of $S_i$ such that, for  $B_i=\bigcap\limits_{s\in K_i}(\bigcup\mathcal{F}_{i,s})$, the sets $B_0, B_1$ are disjoint. Now, it is easily seen that, for every $i\in\{0,1\}$, we can choose a finite family $\mathcal{C}_i$ of $\mu$-closed sets such that $B_i=\bigcup\mathcal{C}_i$. Since every $\langle \mu,g\tau_n\rangle$-continuous function is $\langle \tau(\mu), \tau_n\rangle$-continuous, it follows from $\mathbf{GUL(X)}$ that there exists a $\langle \tau(\mu),\tau_n\rangle$-continuous function $f:X\to [0,1]$ such that, for every $i\in\{0,1\}$,  $\bigcup\mathcal{C}_i\in f^{-1}[\{i\}]$. This shows that $\mathbf{UL}(\langle X, \tau(\mu)\rangle)$ is true.
\end{proof} 

Let us establish the following general fact:

\begin{theorem}
\label{s5t2}
$[\mathbf{ZF}]$ For $i\in\{1,2\}$, let $\mathbf{X}_i=\langle X_i, \mu_i\rangle$ be a strong GT space. Let $X=X_1\times X_2$ and 
$$\mu=\{(U\times X_2)\cup(X_1\times V): U\in\mu_1\text{ and } V\in\mu_2\}.$$ 
Then $\mathbf{X}=\langle X, \mu\rangle$ is a strong GT space. Moreover, the following conditions are satisfied:
\begin{enumerate}
\item[(i)] if, for each $i\in\{1,2\}$, $\mathbf{GUL}(\mathbf{X}_i)$ holds, then $\mathbf{GUL}(\mathbf{X})$ holds;
\item[(ii)] if, for each $i\in\{1,2\}$, $\mathbf{UL}(\mathbf{X}_i)$ holds, then $\mathbf{UL}(\mathbf{X})$ holds;
\item[(iii)] if, for each $i\in\{1,2\}$, $\mathbf{X}_i$ is effectively normal, so is $\mathbf{X}$.
\end{enumerate}
\end{theorem}

\begin{proof} It is obvious that $\mu$ is a strong generalized topology in $X$.  To prove (i)--(iii), let us assume that $A, B$ is a pair of disjoint $\mu$-closed sets. Then  there exist $\mu_1$-closed sets $A_1,B_1$ and $\mu_2$-closed sets $A_2,B_2$ such that $A=A_1\times A_2$ and  $B=B_1\times B_2$. Without loss of generality, we may assume that $A_1\cap B_1=\emptyset$.

To prove (i), assuming that, for each $i\in\{1,2\}$,  $\mathbf{GUL}(\mathbf{X}_i)$ holds, we can take a $\langle\mu_1, g\tau_n\rangle$-continuous function $f: X_1\to\mathbb{R}$ such that $A_1\subseteq f^{-1}[\{0\}]$ and $B_1\subseteq f^{-1}[\{1\}]$. We define a function $g:X\to\mathbb{R}$ by putting $g(x,y)=f(x)$ for every point $\langle x,y\rangle\in X$. The function $g$ is $\langle \mu, g\tau_n\rangle$-continuous, $A\subseteq g^{-1}[\{0\}]$ and $B\subseteq g^{-1}[\{1\}]$. Hence $\mathbf{GUL}(\mathbf{X})$ holds. The proof to (ii) is similar.

To prove (iii), we assume that the spaces $\mathbf{X}_i$ are effectively normal. For each $i\in\{1,2\}$, choose a function $F_i:\mathcal{E}(\mathbf{X}_i)\to\mathcal{O}(\mathbf{X}_i)$ such that, for every $\langle C, D\rangle\in\mathcal{E}(\mathbf{X}_i)$, if $F_i(\langle C, D\rangle)=\langle U, V\rangle$, then $C\subseteq U$ and $D\subseteq V$. Now, we define a function $F:\mathcal{E}(\mathbf{X})\to\mathcal{O}(\mathbf{X})$ as follows. If $A_1\cap B_1=\emptyset$ and $F_1(\langle A_1, B_1\rangle)=\langle U, V\rangle$, we define $F(\langle A, B\rangle)=\langle U\times X_2, V\times X_2\rangle$. If $A_2\cap B_2=\emptyset$ and $F_2(\langle A_2, B_2\rangle)=\langle U, V\rangle$, we define $F(\langle A, B\rangle)=\langle X_1\times U, X_1\times V\rangle$. This shows that $\mathbf{X}$ is effectively normal.
\end{proof}

\begin{example}
\label{s5e1}
$[\mathbf{ZF}]$ Let us consider the following generalized topology $g\tau_s$ in $\mathbb{R}$:
$$g\tau_s=g\tau_n\cup\{[a, +\infty): a\in\mathbb{R}\}\cup\{(-\infty, a)\cup [b, +\infty): a,b\in\mathbb{R}\text{ and } a<b\}.$$
Then $\tau(g\tau_s)$ is the topology of the Sorgenfrey line. Let $X=\mathbb{R}\times\mathbb{R}$ and 
$$\mu=\{(U\times\mathbb{R})\cup(\mathbb{R}\times V): U,V\in g\tau_s\}.$$
Then $\mu$ is a Hausdorff strong generalized topology in $X$.  To check that the GT space $\langle \mathbb{R}, g\tau_s\rangle$ is effectively normal, for every pair $A,B$ of disjoint $g\tau_s$-closed sets such that either $A$ is $g\tau_s$-open or $B$ is $g\tau_s$-open, we put
$$
F(\langle A, B\rangle)=\begin{cases} \langle A, \mathbb{R}\setminus A\rangle &\text{if } A \text{ is } g\tau_s\text{-open};\\
\langle \mathbb{R}\setminus B, B\rangle &\text{otherwise}.
\end{cases}
$$
If $A,B$ is a pair of disjoint $g\tau_s$-closed sets such that neither $A$ nor $B$ is $g\tau_s$-open, we can define $F(\langle A, B\rangle)$ in much the same way, as in the proof to Proposition \ref{s3p5}. In this way, we define a mapping $F:\mathcal{E}(\langle \mathbb{R}, g\tau_s\rangle)\to\mathcal{O}(\langle \mathbb{R}, g\tau_s\rangle)$ witnessing that $\langle \mathbb{R}, g\tau_s\rangle$ is effectively normal.  It follows from Theorem \ref{s3t4} that  $\mathbf{GUL}(\langle \mathbb{R}, g\tau_s\rangle)$ is true. One can also check that $\mathbf{GUL}(\langle \mathbb{R}, g\tau_s\rangle)$ holds by a slight modification of the proof to Proposition \ref{s2p3}. It follows from Theorem \ref{s5t2} that $\mathbf{GUL}(\langle X, \mu\rangle)$ is true. That $\mathbf{UL}(\langle X, \tau(\mu)\rangle)$ is not true follows from the well-known fact that the square of the Sorgenfrey line is not normal in $\mathbf{ZF}$. 
\end{example}

\end{document}